\documentstyle{article}

\newtheorem{theorem}{Theorem}[subsection]

\newtheorem{lemma}[theorem]{Lemma}

\newcommand{\qed}{\rule{1mm}{3mm}}     
\newenvironment{proof}{\vspace*{\parsep}\noindent {\bf proof:}}{\qed\\[1em]}
\begin{document}

\title{Independence Results around Constructive ZF}
\author{Robert S. Lubarsky}
\maketitle

\section{Introduction}

When making the transition from a classical to an intuitionistic
system, one is compelled to reconsider the axioms chosen. For
instance, it would be natural to re-formulate those that imply
Excluded Middle. In the context of set theory, typically the Axiom
of Foundation is substituted by the Axiom of Set Induction, the
two being equivalent classically but the latter not yielding
Excluded Middle intuitionistically. Even axioms that do not return
classical logic must be re-evaluated, as classically equivalent
variants may not remain so in a different logic. An example of
this is Collection, which, over the other ZF axioms, is equivalent
to Replacement and to Reflection, but not over intuitionistic ZF
(see \cite{FS}).

A different reason to re-evaluate axioms is to take more seriously
the program of constructivism. IZF is basically ZF with classical
logic replaced by intuitionistic. Since, though, the axioms of ZF
were never meant to be constructive, they're still not even with a
different logic. This issue was addressed by Myhill (see
\cite{my}), who re-formulated by the axioms and framework of ZF to
come up with CST (Constructive Set Theory). Among the changes he
made was to eliminate the Power Set Axiom as non-constructive; as
an alternative, he included Exponentiation, the hypothesis of the
existence of the set of functions from a set A to a set B, as
adequately strong to do most of ordinary mathematics yet still
sufficiently constructive.

Later on, Aczel (\cite{ac1}, see also \cite{ac2} and \cite{ac3})
picked up the thread and re-cast CST in a more conventional
framework, christening this revamped version CZF (Constructive
ZF). He also provided a constructive justification for the
enterprise, by interpreting CZF in Martin-L\"of type theory
(\cite{ml1}, \cite{ml2}). In so doing, he uncovered a principle
apparently stronger than Exponentiation, Subset Collection. Subset
Collection is most easily understood via an equivalent, Fullness.
A set C of total relations from A to B is said to be full if every
total relation from A to B can be thinned to one in C. The
Fullness Axiom says that every pair of sets has a full set of
relations. This principle is valid in Aczel's interpretation, and
so is included in CZF.

This all leads to asking what the exact relationship among these
axioms actually is. It is easily seen that Power Set implies
Subset Collection, which itself implies Exponentiation. What about
the reverse implications? SC was chosen over Exponentiation
because it's presumably stronger, and Power Set not because it's
presumably not Martin-L\"of interpretable. Presumably. In this
paper, we show that the presumed non-implications are indeed
correct. In the next section, it is show that SC does not imply
Power Set, and in the one after that, that Exponentiation does not
imply SC. In the final section, we observe that a result about
Reflection is a by-product of the earlier work.

In fact, although these questions came up in the context of CZF,
slightly more than independence over the remaining CZF axioms is
shown. The other difference between CZF and IZF is that CZF
contains Restricted Separation (i.e. Separation for $\Delta_{0}$
definable properties), whereas IZF has full Separation. The
independece results below include full Separation, and so are
actually independence results over IZF - Power Set.

For the record, the axioms of IZF are: Empty, Infinity, Pair,
Union, Extensionality, Set Induction, Power Set, Separation, and
Collection. CZF consists of Empty, Strong Infinity, Pair, Union,
Extensionality, Set Induction, Subset Collection, $\Delta_{0}$
Separation, and Strong Collection. Strong Infinity is a variant of
Infinity that more directly produces $\omega$ without resort to
Power Set or full Separation. Similarly, Strong Collection is a
variant of Collection useful in the absence of full Separation.
Collection asserts the existence of a bounding set (for a total
function on a set defined by a formula). In the presence of
Separation, one can take those elements of a postulated bounding
set that come from something in the domain. Strong Collection is
the assertion that there is a bounding set consisting only of
range values for something in the domain.

\section{Subset Collection does not imply Power Set}
We will provide a Kripke model of CZF + $\neg$ Power Set (within
the meta-theory ZF). Furthermore, full Separation will hold.

Given a partial order {\it P}, the full Kripke model {\bf M} on
{\it P} is defined inductively on $\alpha \in$ ORD as follows. At
node $\sigma$, the universe {\bf M}$^{\alpha}_{\sigma}$ consists
of those functions f with domain {\it P}$_{\geq \sigma}$ such that
for $\tau \geq \sigma$ f($\tau$) $\subseteq \bigcup_{\beta <
\alpha}$ {\bf M}$^{\beta}_{\tau}$ and for $\rho \geq \tau$
f$_{\tau\rho}$"f($\tau$) $\subseteq$ f($\rho$) (where
f$_{\tau\rho}$ is the transition function from node $\tau$ to node
$\rho$). f$_{\sigma\tau}$ is then extended from $\bigcup_{\beta <
\alpha}$ {\bf M}$^{\beta}_{\sigma}$ to {\bf M}$^{\alpha}_{\sigma}$
as restriction: f$_{\sigma\tau}$(f) = f $\uparrow$ {\it P}$_{\geq
\tau}$. The universe at $\sigma$, {\bf M}$_{\sigma}$, is defined
as $\bigcup_{\alpha \in ORD}$ {\bf M}$^{\alpha}_{\sigma}.$ Then
$\in$ and = are as easy as pie to define: $\sigma \vdash f \in g$
iff $f \in g(\sigma)$, and $\sigma \vdash f = g$ if f = g.

Equivalently, one could
 simplify the inductive generation of the Kripke sets and the transition functions, at the
 expense of $\in$ and =, by having the Kripke sets have domain all of {\it P}:
 {\bf M}$^{\alpha} = \{ f : {\it P} \rightarrow \bigcup_{\beta < \alpha}$ {\bf M}$^{\beta} \mid f$ is
 monotonically non-decreasing (i.e. $\sigma < \tau \rightarrow f(\sigma) \subseteq f(\tau)) \}$.
 Then {\bf M}$_{\sigma}$ is just $\bigcup_{\alpha \in ORD}$ {\bf M}$^{\alpha}$ (which, note, is independent
 of $\sigma$), and the transition functions are just the identity. The complication pops up again when
 defining $\in$ and =, since at $\sigma$ all nodes not extending $\sigma$ are unimportant but still in
 the domains and hence need to be actively disregarded: $\sigma \vdash f \in g$ if
 $\exists h \in g(\sigma) \; \sigma \vdash f = h$, and $\sigma \vdash f = g$ if $\forall
 \tau \geq \sigma \; [\forall h \in f(\tau) \; \tau \vdash h \in g(\tau) \; \wedge \; \forall h \in g(\tau)
 \; \tau \vdash h \in
 f(\tau)]$.

The model we ultimately want to consider differs from that above in that the Kripke sets are to be, as functions, eventually constant. There are several ways this could be achieved. The formalism developed below will be to have {\it P} be a class; in particular, ORD. Since every Kripke set is a set, it can give information about only a finite initial segment of {\it P}, after which it is assumed to be constant. Alternatively, {\it P} could be taken to be $\omega$ and the Kripke sets to be eventually hereditarily constant. In fact, this latter approach could be read off from the former by taking a cofinal $\omega$-sequence through ORD (from outside of {\bf V}, naturally) and cutting the full model down to those nodes. (That this works might be construed as a reminder that $\omega$ is still the largest large cardinal of all.) Neither approach seems essentially easier than the other, so, perhaps arbitrarily, we will henceforth deal only with the former.

At this point it would be prudent, instead of formalizing the ideas above in their full generality, to particularize to the partial order of interest. So let {\it P} be ORD. The nodes in {\it P} will typically be referred to with letters from the middle of the Greek alphabet ($\kappa, \lambda, \mu, \nu$ and so on) so as not to cause confusion with the use of $\alpha, \beta,$ and $\gamma$ used in inductions.

Both constructions above can be adapted to this choice of {\it P}. The adaptation of the second construction will be described informally, followed by a thorough description using the first construction, which we remain the basis of this paper. So inductively a Kripke set is a function from {\it P} = ORD to the class of previously defined Kripke sets, which is not only monotonic but also eventually constant: {\bf M}$^{\alpha} = \{ f : ORD \rightarrow \bigcup_{\beta < \alpha}$ {\bf M}$^{\beta} \mid f$ is monotonically non-decreasing and $\exists \kappa \forall \lambda > \kappa \; f(\lambda) = f(\kappa) \}$. Since f($\kappa$) is set-sized, not only is each of its elements eventually constant (inductively), but there is a uniform point of stability: $\exists \kappa \forall g \in TC(f)$ (transitive closure) g is constant beyond $\kappa$. The transition functions, as well as membership and equality, as are above.

To formalize this as stated, an axiomatic theory of classes would be necessary. To work within a more conventional framework, meaning a theory of sets (the framework of this paper being, after all, ZF), an element f will have to be a function on a (necessarily) set-sized initial segment of ORD, beyond which f is implicitly constant. The technical work here is to define the transition functions f$_{\kappa\lambda}$ on f, even when $\lambda \not \in$ dom(f). Since this involves mucking with the domains anyway, it's easier to adapt the former of the constructions above. The reader with a clear intuition who is willing to trust that these matters can be formalized may safely skip to the proposition.

The full Kripke model {\bf M} on {\it P} is defined inductively on $\alpha \in$ ORD as follows. At node $\kappa$, the universe {\bf M}$^{\alpha}_{\kappa}$ consists of those functions f with domain $\lambda \backslash \kappa$ (for some $\lambda > \kappa$) such that for $\mu \in$ dom(f) f($\mu$) $\subseteq \bigcup_{\beta < \alpha}$ {\bf M}$^{\beta}_{\mu}$ and for $\nu \in$ dom(f), $\nu \geq \mu$, f$_{\mu\nu}$"f($\mu$) $\subseteq$ f($\nu$) (where f$_{\mu\nu}$ once again is the transition function from node $\mu$ to node $\nu$). Without loss of generality we can also require that for all $\mu \in$ dom(f) and g $\in$ f($\mu$) dom(f) $\supseteq$ dom(g). (``WLOG" means here that this additional stipulation does not exclude any sets. More precisely, if f satisfies the definition for a set in the Kripke structure except for the condition on the domain, then there is a Kripke set g satisfying the domain condition such that $\kappa \models$ ``f=g". Such a g would be constructed as an end-extension of f, whereby for $\mu \in$ dom(g)$\backslash$dom(f), g($\mu$) would be taken as f$_{\kappa\mu}$(f), using the transition functions about to be defined.)

f$_{\kappa\mu}$ is then extended from $\bigcup_{\beta < \alpha}$ {\bf M}$^{\beta}_{\kappa}$ to {\bf M}$^{\alpha}_{\kappa}$ as follows. If $\mu \in$ dom(f), then f$_{\kappa\mu}$ is restriction: f$_{\kappa\mu}$(f) = f $\uparrow$ dom(f)$\backslash \mu$. If $\mu \not \in$ dom(f), then f$_{\kappa\mu}$(f) will have domain \{$\mu$\} (in the notation from the last paragraph, $\lambda$ will be $\mu + 1$), and (f$_{\kappa\mu}$(f))($\mu$) will be (working inductively here) $\bigcup_{\nu \in dom(f)}$ f$_{\nu\mu}$"f($\nu$). It is left to the reader to show that the transition functions compose as they are supposed to (i.e. f$_{\lambda \mu} \circ$ f$_{\kappa \lambda}$ = f$_{\kappa \mu}$).

The universe at node $\kappa$, {\bf M}$_{\kappa}$, can now be defined as $\bigcup_{\alpha \in ORD}$ {\bf M}$^{\alpha}_{\kappa}.$ The transition functions f$_{\kappa\mu}$ defined along the way on {\bf M}$^{\alpha}_{\kappa}$ cohered, which is why the dependence on $\alpha$ was dropped from the notation; the ultimate f$_{\kappa\mu}$ defined on {\bf M}$_{\kappa}$ is the union of the partial f$_{\kappa\mu}$'s defined along the way.

We will need that the model {\bf M} is definable in {\bf V}, so we may as well check each building block in the development of {\bf M} when it's first introduced. Notice that {\bf M}$^{\alpha}_{\kappa}$ cannot be a function of $\alpha$, since {\bf M}$^{\alpha}_{\kappa}$ is not even a set! However, the relation ``f $\in$ {\bf M}$^{\alpha}_{\kappa}$" (as a relation on f, $\alpha$, and $\kappa$), as well as ``f$_{\kappa\mu}$(f) = g" (as a relation on f, g, $\kappa,$ and $\mu$) is $\Delta_{1}$. This can be shown via a simultaneous $\Sigma_{1}$ induction. The central point in this argument is that the definitions (as given above) contain only bounded quantification and recursive calls.

$\in$ and = must then be defined via a mutual induction to make sure that Extensionality holds in the end. The problem here, which does not come up in the case of a set-sized underlying p.o., is illustrated by the fact that both $\{ \langle 0, \emptyset \rangle \}$ and $\{ \langle 0, \emptyset \rangle, \langle 1, \emptyset \rangle \}$ stand for the empty set yet are not themselves equal as sets. The reader uninterested in such details can safely skip to the main proposition.

Assume that $\in_{\beta}$ and =$_{\beta}$ are defined on {\bf M}$^{\beta}_{\kappa}$ for all nodes $\kappa$ and, inductively, for all stages $\beta < \alpha$. For f, g $\in$ {\bf M}$^{\alpha}_{\kappa}$, $\kappa \models$ ``f $\in_{\alpha}$ g" iff for some h $\in$ g($\kappa$) and $\beta < \alpha, \; \kappa \models$ ``f=$_{\beta}$h". Notice that the sequence of $\in_{\alpha}$s is monotonically non-decreasing, that is, if f $\in_{\alpha}$ g and $\gamma > \alpha$ then f $\in_{\gamma}$ g. This holds because the defining condition of $\in_{\alpha}$ remains true if $\alpha$ is increased. Regarding =, $\kappa \models$ ``f=$_{\alpha}$g" iff for all $\lambda \in$ dom(f) $\cup$ dom(g), h $\in$ (f$_{\kappa \lambda}$(f))($\lambda$), $\lambda \models$ ``h $\in_{\alpha}$ f$_{\kappa \lambda}$(g)", and for all h $\in$ (f$_{\kappa \lambda}$(g))($\lambda$), $\lambda \models$ ``h $\in_{\alpha}$ f$_{\kappa \lambda}$(f)". Notice that the sequence of =$_{\alpha}$s is non-decreasing, because, in the defining condition, $\alpha$ appears only in $\in_{\alpha}$, $\in_{\alpha}$ appears only positively, and $\in_{\alpha}$ is already known to be non-decreasing. Furthermore, $\in_{\alpha}$ and =$_{\alpha}$ are non-increasing for a fixed collection of sets. That is, for $\alpha$ = max(rk f, rk g), $\kappa \models$ ``f $\in_{\alpha}$ g" iff for some $\gamma > \alpha$ $\kappa \models$ ``f $\in_{\gamma}$ g", and similarly for =$_{\alpha}$. This can be proved by a mutual induction. Finally, let $\in$ be $\bigcup_{\beta \in ORD} \; \in_{\beta}$, i.e. $\kappa \models$ ``f $\in$ g" iff for some $\beta \; \kappa \models$ ``f $\in_{\beta}$ g", and similarly for =.

Considering definability again, both $\in$ and = are $\Delta_{1}$-definable ternary relations (i.e. $\kappa \models$ f $\in$ g and $\kappa \models$ f = g). This is based on two facts. First, for $\alpha$ = max(rk f, rk g), $\kappa \models$ ``f $\in$ g" iff $\kappa \models$ ``f $\in_{\alpha}$ g", and similarly for =. Second, $\in_{\alpha}$ and =$_{\alpha}$ are defined by a simultaneous recursion which contains only bounded quantification and recursive calls.

\begin{lemma}
The transition functions respect the primitive relations $\in$ and =.
\end{lemma}

\begin{proof}
It suffices to show that the transition functions respect
$\in_{\alpha}$ and =$_{\alpha}$. This can be done by a
simultaneous induction on $\alpha$. The case of $\in_{\alpha}$ is
trivial. To show $\kappa \models$ f=$_{\alpha}$g implies $\mu
\models$ f$_{\kappa \mu}$(f)=$_{\alpha}$f$_{\kappa \mu}$(g), if
$\mu \in$ dom(f) $\cup$ dom(g), then this is trivial, since the
only effect of f$_{\kappa \mu}$ is to shrink the set of
$\lambda$'s that need be considered. Otherwise one has to unravel
the definition of f$_{\kappa \mu}$, which again brings us to
consider those $\lambda$s in dom(f) $\cup$ dom(g). The details are
left to the reader.
\end{proof}

\begin{lemma}
= (as defined above) satisfies the equality axioms.
\end{lemma}

\begin{proof}
1. 0 $\models \forall x \; x=x$: Assume inductively that for all x
$\in$ {\bf M}$^{\beta}_{\kappa}$ ($\beta < \alpha$) $\kappa
\models$ x=$_{\beta}$x and that f $\in$ {\bf
M}$^{\alpha}_{\kappa}$. We need to show that $\kappa \models$
f=$_{\alpha}$f, i.e. for all $\lambda \in$ dom(f), h $\in$
(f$_{\kappa \lambda}$(f))($\lambda$), $\lambda \models$ ``h
$\in_{\alpha}$ f$_{\kappa \lambda}$(f)". Since $\lambda \in$
dom(f), f$_{\kappa \lambda}$ operates via restriction, so the
condition to be proved simplifies to for all $\lambda \in$ dom(f),
h $\in$ f($\lambda$), $\lambda \models$ ``h $\in_{\alpha}$
f$_{\kappa \lambda}$(f)". Given such a $\lambda$ and h, by the
definition of $\in_{\alpha}$, we need a g $\in$ f($\lambda$) and
$\beta$ so that $\lambda \models$ ``h=$_{\beta}$g". Letting g be h
and $\beta$ be rk(h), the inductive hypothesis gives exactly this.

2. 0 $\models \forall x,y \; x=y \rightarrow y=x$: Trivial,
because the definition of $\in_{\alpha}$ was symmetric in f and g.

3. 0 $\models \forall x,y,z \; x=y \wedge y=z \rightarrow x=z$: As
in 1, the transitivity of =$_{\alpha}$ is proven inductively on
$\alpha$. The technical difficulty here is that we must consider
the domains of arbitrary Kripke sets f, g, and h, which might all
differ. Possibly the easiest way to do this is to start with the
observation that, for arbitrary $\lambda \geq \kappa$, (f$_{\kappa
\lambda}$(f))($\lambda$) = $\bigcup_{\mu \leq \lambda, \mu \in
dom(f)}$ f$_{\mu \lambda}$"f($\mu$). So when evaluating $\kappa
\models$ f=$_{\alpha}$h, after being led to consider i $\in$
(f$_{\kappa \lambda}$(f))($\lambda$), we can conclude that i $\in$
f$_{\mu \lambda}$"f($\mu$) for some $\mu \in$ dom(f). Using a
pre-image i$_{\mu}$ of i in f($\mu$) and the postulated equality
of f and g, one finds a $\nu \in$ dom(g), $\nu \leq \mu$, a
pre-image i$_{\nu}$ of i$_{\mu}$, and a j$_{\nu} \in$ g(${\nu}$)
such that i$_{\nu}$ =$_{\beta}$ j$_{\nu}$. Then one finds a $\xi
\leq \nu$, j$_{\xi} \in$ g(${\xi}$), k$_{\xi} \in$ h($\xi$) such
that j$_{\xi}$ =$_{\gamma}$ k$_{\xi}$. At this point one uses the
monotonicity of = to move up to $\delta$ = max($\beta, \gamma$),
the inductively assumed transitivity of =$_{\delta}$ to get that
i$_{\xi}$ =$_{\delta}$ k$_{\xi}$, and the fact that the transition
functions respect the primitive relations to conclude $\lambda
\models$ ``i $\in_{\alpha}$ h". The details are left to the
reader.

4. 0 $\models \forall x,y,z \; x=y \wedge x \in z \rightarrow y
\in z$: If $\kappa \models$ ``f=g $\wedge$ f $\in$ h", then
$\kappa \models$ ``f=i" for some i $\in$ h($\kappa$). By the
transitivity of =, $\kappa \models$ ``g=i".

5. 0 $\models \forall x,y,z \; x=y \wedge z \in x \rightarrow z
\in y$: If $\kappa \models$ ``f=g $\wedge$ h $\in$ f", then for
some i $\in$ f($\kappa$) $\kappa \models$ ``h=i", and also $\kappa
\models$ ``i $\in$ g". That means that for some j $\in$
g($\kappa$) $\kappa \models$ ``i=j". By the transitivity of =,
$\kappa \models$ ``h=j", and hence also ``h $\in$ g".

\end{proof}

With $\in$ and = defined, the description of the model {\bf M} is
complete. The definability of {\bf M} follows automatically from
the work above. Since the universes at each node and the primitive
relations (there being no primitive functions) are definable, the
definability of the satisfaction relation $\kappa \models \phi$
comes directly from the inductive determination of satisfaction in
Kripke models. Furthermore, while we make no claim along the lines
of $\Sigma_{n}$-satisfaction being $\Sigma_{n}$ definable, it is
true that $\Delta_{0}$-satisfaction (i.e. $\kappa \models \phi,
\phi$ a $\Delta_{0}$ formula) is $\Delta_{1}$ definable, as
follows. In the inducive definition of satisfaction in Kripke
models (when restricting to $\Delta_{0}$ formulas), the only times
an unbounded quantifier is introduced are for the clauses
$\forall$ and $\rightarrow$, when the collection of nodes is
quantified over. Such quantification can be eliminated by
restricting it to the maximum $\kappa$ of the domains of the
parameters, because the formula is true at $\kappa$ iff it's true
at any later node, this last fact being provable inductively. (A
different proof, which shows the same fact without the restriction
of $\phi$ being $\Delta_{0}$, is that there is an isomorphism
between the Kripke model from $\kappa$ on and the model from
$\lambda$ on, where $\lambda$ is any ordinal $> \kappa$. This
latter argument is given in detail in the proof of Separation
later in this section.)

What remains is to show that {\bf M} has the desired properties.

\begin{theorem}
{\bf M} $\models$ CZF + $\neg$ Power Set
\end{theorem}

\begin{proof}
The simpler axioms of CZF are Pairing, Union, and Infinity, and
the proofs that they hold in {\bf M} are left to the reader. What
we will show are Extensionality, Set Induction, Restricted
Separation, Subset Collection, Strong Collection, and, of course
$\neg$ Power Set. As promised above, we will go beyond CZF and
show full Separation.

Extensionality: What needs to be shown is $\forall x \; \forall y
\; (x=y \leftrightarrow \forall z \; (z \in x \leftrightarrow z
\in y))$. The implication $\rightarrow$ follows easily from
equality property 5. in the previous proposition. For the
implication $\leftarrow$, given sets f and g and node $\kappa$,
let $\alpha$ = max(rk f, rk g). To show that f=$_{\alpha}$g is
precisely, by definition, $\forall \lambda \in$ dom(f) $\cup$
dom(g), h $\in$ (f$_{\kappa \lambda}$(f))($\lambda$), $\lambda
\models $ ``h $\in_{\alpha}$ f$_{\kappa \lambda}$(g)", and vice
versa. So let h be a member of (f$_{\kappa
\lambda}$(f))($\lambda$). By assumption, we have $\kappa \models
``\forall z \; (z \in f \leftrightarrow z \in g))"$. This holds in
particular for h: $\lambda \models $ ``h $\in (f_{\kappa
\lambda}$(f)) $\leftrightarrow$ h $\in$ (f$_{\kappa
\lambda}$(g))". In order to use this, we need that $\lambda
\models $ ``h $\in$ (f$_{\kappa \lambda}$(f))", when what we have
is merely h $\in$ (f$_{\kappa \lambda}$(f))($\lambda$). Unraveling
the definition of $\lambda \models x \in_{\beta} y$ and making the
obvious choices along we way, this reduces to showing that
$\lambda \models $ h=$_{\beta}$h for some $\beta$. Letting $\beta$
be rk(h), this follows from the proof (albeit not the statement)
of the first equality axiom (previous proposition). Hence we can
conclude that $\lambda \models$ ``h $\in$ (f$_{\kappa
\lambda}$(f)) $\leftrightarrow$ h $\in$ (f$_{\kappa
\lambda}$(g))". What we still need is that $\lambda \models $ ``h
$\in_{\alpha}$ f$_{\kappa \lambda}$(g)". That this is the case was
already observed while defining $\in_{\alpha}$ and =$_{\alpha}$.
The other direction (``vice versa") is symmetric in f and g. The
details are left to the reader.

Set Induction: This is easy, because the Kripke sets were defined inductively. That is, given an appropriate property $\phi(x)$, one would show that $\phi(x)$ holds for all members of {\bf M}$^{\alpha}$ inductively on $\alpha$.

Restricted Separation: In fact, full Separation holds. Let $\kappa$ be a node, $a$ a set in {\bf M}$_{\kappa}$, and $\phi$ a formula (with free variable $x$) over (i.e. with parameters from) {\bf M}$_{\kappa}$. Let $\lambda > \kappa$ include not only dom($a$) but also the domains of all of $\phi$'s parameters. For $\mu > \kappa$ we will use the notation f$_{\kappa\mu}(\phi)(x)$ to mean $\phi(x$, f$_{\kappa\mu}(b_{0})$, ... , f$_{\kappa\mu}(b_{n-1}))$, where the $b_{i}$'s list all of $\phi$'s parameters. In {\bf V}, let Sep$_{a, \phi}$ be a function with domain $\lambda + 1 \backslash \kappa$ such that for all $\mu$ in this domain
$$Sep_{a, \phi}(\mu) = \{ x \in (f_{\kappa\mu}(a))(\mu) \mid \mu \models f_{\kappa\mu}(\phi)(x)\}.
$$
This function Sep$_{a, \phi}$ is definable in {\bf V} because {\bf M} is, and it should be clear that for all $\mu \in$ dom(Sep$_{a, \phi}$), $x \in$ {\bf M}$_{\mu}$,
$$\mu \models ``x \in f_{\kappa\mu}(Sep_{a, \phi}) \leftrightarrow x \in f_{\kappa\mu}(a) \wedge f_{\kappa\mu}(\phi)(x)".
$$
We need only show why the same holds also for $\mu > \lambda$. The idea is that beyond $\lambda$ nothing relevant (i.e. $a$ or $\phi$) is changing anymore, and so Sep$_{a, \phi}$ need not change either. In detail, for any such $\mu$, by the choice of $\lambda$, f$_{\lambda\mu}$ is an isomorphism on f$_{\kappa\lambda}(a)$, f$_{\kappa\lambda}(\phi)$, and f$_{\kappa\lambda}$(Sep$_{a, \phi}$), the nature of which is replacing the domains in the transitive closures of the sets in question, which are all \{$\lambda$\}, with \{$\mu$\}. This isomorphism can be extended to one between {\bf M}$_{\lambda}$ and  {\bf M}$_{\mu}$ by hereditarily replacing ordinals in the domains, necessarily of the form $\lambda + \xi$, with $\mu + \xi$. Under this isomorphism, \{$x \in {\bf M}_{\lambda} \mid \lambda \models ``x \in f_{\kappa\lambda}(a) \wedge f_{\kappa\lambda}(\phi)(x)"$\}, i.e. Sep$_{a, \phi}$($\lambda$),  gets mapped to \{$x \in {\bf M}_{\mu} \mid \mu \models ``x \in f_{\lambda\mu}(f_{\kappa\lambda}(a)) \wedge f_{\lambda\mu}(f_{\kappa\lambda}(\phi))(x)"$\}. By the composition property of transition functions (f$_{\sigma\tau} \circ$ f$_{\rho\sigma}$ = f$_{\rho\tau}$, a necessary condition on all Kripke models), we have the desired result.

Subset Collection: Let $\kappa \in$ ORD, $a, b \in {\bf M}_{\kappa}$. Let $\lambda \supseteq$ dom($a$), dom($b$). SubColl$_{a,b}$ will have domain $\lambda + 1 \backslash \kappa$. For all $\mu$ in this domain, let
$$SubColl_{a,b}(\mu) = \{ x \in {\bf M}_{\mu} \mid \mu \models ``x \; is \; a \; total \; relation \; from \; a \; to \; b"\}.
$$
It is immediate that for all $\mu \in$ dom(SubColl$_{a,b}$), $x \in$ {\bf M}$_{\mu}$,
$$\mu \models ``x \in f_{\kappa\mu}(SubColl_{a,b}) \leftrightarrow x \; is \; a \; total \; relation \; from \; a \; to \; b".
$$
For $\mu > \lambda$, the idea is that from $\lambda$ to $\mu$ neither $a$ nor $b$ has changed. So if node $\mu$ forces $x$ to be a total relation from $a$ to $b$, then any ordered pairs that might enter $x$ at a later node can be disregarded and $x$ would still remain total, since $a$ doesn't grow any more. Furthermore, the ordered pairs at node $\mu$ can be ``pulled back" to node $\lambda$, since neither $a$ nor $b$ grew from $\lambda$ to $\mu$. This (possibly truncated) version of $x$ was put into SubColl$_{a,b}$ by at latest node $\lambda$, and is indeed a total sub-relation of $x$.

In more detail, observe that for any such $\mu$, by the choice of $\lambda$, f$_{\lambda\mu}$ is an isomorphism on f$_{\kappa\lambda}(a)$ and f$_{\kappa\lambda}(b)$. Let $\mu \models ``x$ is a total relation from f$_{\kappa\mu}(a)$ to f$_{\kappa\mu}(b)$". Then $x(\mu)$ is a total relation from (f$_{\kappa\mu}(a)$)($\mu$) to (f$_{\kappa\mu}(b)$)($\mu$). This induces via the isomorphism a total relation from (f$_{\kappa\lambda}(a)$)($\lambda$) to (f$_{\kappa\lambda}(b)$)($\lambda$); let $z$ be the function with domain \{$\lambda$\} and $z(\lambda)$ being this relation. $\lambda \models ``z$ is total" since $a$ doesn't change beyond $\lambda$, so $\lambda \models ``z \in$ SubColl$_{a,b}$". By the construction of $z$, $\mu \models f_{\lambda\mu}(z) \subseteq x$.

Strong Collection: Let $\kappa$ be a node, $a$ a set in {\bf M}$_{\kappa}$, and $\phi$ a formula (with free variable $x$) over (i.e. with parameters from) {\bf M}$_{\kappa}$, such that $\kappa \models ``\forall x \in a \; \exists y \; \phi(x,y)$". Let $\lambda > \kappa$ include not only dom($a$) but also the domains of all of $\phi$'s parameters. As above, the notation f$_{\kappa\mu}(\phi)(x)$ means $\phi(x$, f$_{\kappa\mu}(b_{0})$, ... , f$_{\kappa\mu}(b_{n-1}))$, where the $b_{i}$'s list all of $\phi$'s parameters. In {\bf V}, a function StrColl$_{a, \phi}$ with domain $\lambda + 1 \backslash \kappa$ can be constructed so that for all $\mu$ and $\nu$ in this domain with $\mu < \nu$,
$$\forall x \in a(\mu) \; \exists y \in StrColl_{a, \phi}(\mu) \; \mu \models \phi(x,y),
$$
$$\forall y \in StrColl_{a, \phi}(\mu) \; \exists x \in a(\mu) \; \mu \models \phi(x,y),
$$
and f$_{\mu\nu}$"StrColl$_{a, \phi}(\mu) \subseteq$ StrColl$_{a, \phi}(\nu)$.
StrColl$_{a, \phi}$ clearly acts as a bounding set at all nodes through $\lambda$. Beyond $\lambda$ the idea is that nothing relevant (i.e. $a$ or $\phi$) is changing anymore, and so StrColl$_{a, \phi}$need not change either. The argument is essentially the same as that for Restricted Separation, and is left to the reader.

$\neg$ Power Set: For $\kappa \in$ ORD, let 1$_{\kappa}$ be the Kripke set with domain $\kappa + 1$ such that for $\lambda < \kappa \; 1_{\kappa}(\lambda) = \emptyset$ and $1_{\kappa}(\kappa) = \langle \kappa, \emptyset \rangle$. $0 \models ``1_{\kappa} \subseteq 1 =_{def} \{0\}"$, but by the set-sized nature of the members of {\bf M}, no set contains all of the 1$_{\kappa}$'s. Therefore 1 has no power set in {\bf M}.
\end{proof}

\section{Exponentiation does not imply Subset Collection}
We will provide a Kripke model of CZF - Subset Collection + Exponentiation + $\neg$ Subset Collection. Furthermore, full Separation will hold.

The model to be built in this section is an extension of that from the previous section, and ultimately can be defined only in a generic extension of {\bf V}. These sets, though, will be named by terms in {\bf V} and so we will work in {\bf V} at first to identify the terms. When the terms are interpreted, they will be functions just as in the last section, and so their domains will also be given during their inductive generation.

The terms are defined inductively on the ordinals. At stage 0, at any node $\kappa$, there is a term f$_{0\kappa}(G_{\lambda})$, which will be interpreted as a function with domain $[\kappa,$ max($\kappa, \lambda)]$. (For $\kappa$ = 0, these will be referred to more simply as $G_{\lambda}$.) (Of course, in this context the symbol ``f$_{0\kappa}$" is just a piece of syntax. It should, however, remind you of the transition function from node 0 to $\kappa$. In an attempt to reduce notation, we will use the same notation for a term as for its interpretation, hence the double meaning of ``f$_{0\kappa}$".) At any later stage, if $a$ is a term at node $\kappa$, then so is Union($a$) (with the same domain); if $a$ is a term at node $\kappa$, $\phi$ is a $\Delta_{0}$ formula (in which all of the parameters are terms at node $\kappa$), and $\lambda > \kappa$, then Sep$_{a,\phi,\lambda}$ is a term at node $\kappa$ with domain $[\kappa, \lambda)$; and if a function is build via the inductive construction of the previous section (given at the bottom of p. 3), then it is also a term in which the node and domain of definition are explicit.

Concerning the interpretation of these terms, we introduce some notation. For an integer n and node $\kappa$, n$_{\kappa}^{\bf M}$ is the internal version of n; that is, 0$_{\kappa}^{\bf M}$ = $\langle \kappa, \emptyset \rangle$ and n+1$_{\kappa}^{\bf M}$ = $\langle \kappa, n_{\kappa}^{\bf M} \cup \{n_{\kappa}^{\bf M}\} \rangle$. Let {\bf N}$_{\kappa}^{\bf M}$ be \{n$_{\kappa}^{\bf M} \mid$ n $\in {\bf N}$\}. (Notice that {\bf N}$_{\kappa}^{\bf M}$ is not actually a set {\bf M}; if we had instead wanted to define the internal version of the natural numbers, say, that would be $\langle \kappa, {\bf N}_{\kappa}^{\bf M} \rangle$ (or a variant thereof).)

Let G be a generic relation on {\bf N}. That is, G is generic over {\bf V} for the forcing partial order which consists of functions from finite subsets of {\bf N} $\times$ {\bf N} to 2; p extends q (written ``p $\geq$ q") if p extends q as a function, i.e. p $\supseteq$ q. A generic can be identified with a total function F : {\bf N} $\times$ {\bf N} $\rightarrow$ 2; letting G be \{ $\langle m, n \rangle \mid$ F($\langle m, n \rangle$) = 1 \}, G is then a generic relation on {\bf N}. Notice that by genericity, both G and G$^{-1}$ are total.

In {\bf V}[G], let G$_{\kappa}^{\bf M}$ be \{ $\langle m_{\kappa}^{\bf M}, n_{\kappa}^{\bf M} \rangle^{\bf M} \mid \langle m, n \rangle \in$ G \} (where $\langle ,\rangle^{\bf M}$ means the {\bf M}-internal ordered pair function) . Similarly, let ({\bf N}$\times${\bf N})$_{\kappa}^{\bf M}$ be \{ $\langle m_{\kappa}^{\bf M}, n_{\kappa}^{\bf M} \rangle^{\bf M} \mid m, n \in {\bf N}_{\kappa}^{\bf M}$  \}.

Using the G$_{\kappa}^{\bf M}$s, we can now provide an interpretation for the terms. (Remember, we are using the same notation for the interpretation of a term as for the term itself; whether a term or its interpretation is meant should be clear from the context.) For f$_{0\kappa}(G_{\lambda})$, f$_{0\kappa}(G_{\lambda})(\mu)$ = G$_{\mu}^{\bf M}$ for $\mu < \lambda$ (the internal version of G), and f$_{0\kappa}(G_{\lambda})(\lambda)$ = ({\bf N}$\times${\bf N})$_{\kappa}^{\bf M}$ otherwise (the full relation on {\bf N}). Union($a$)($\kappa$) (where $\kappa \in$ dom($a$)) is $\bigcup_{f \in a(\kappa)} f(\kappa)$. The interpretation of Sep$_{a,\phi,\lambda}$ is similar to that of Sep$_{a,\phi}$ in the previous section (See the proof of Restricted Separation), only there an essentially arbitrary $\lambda$ was chosen (the only requirement being that $\lambda$ be big enough), whereas here the $\lambda$ in question is a part of the term and the size restriction is dropped. So for $\mu \in [\kappa, \lambda),$
$$Sep_{a,\phi,\lambda}(\mu) = \{ x \in (f_{\kappa\mu}(a))(\mu) \mid \mu \models f_{\kappa\mu}(\phi)(x)\}.
$$
The interpretation of a function is straightforward: mod the above mentioned identification of terms with their interpretations, there is literally nothing to do.

Of course, to have a Kripke model, we need some other facts and definitions; indeed, the interpretation above is dependent upon some of these, which would therefore need to be proven and defined simultaneously with this inductive interpretation. That is, we need to define transition functions f$_{\kappa\mu}$, show that they cohere (i.e. f$_{\lambda \mu} \circ$ f$_{\kappa \lambda}$ = f$_{\kappa \mu}$), and show that the Kripke sets are non-decreasing (i.e. if $\kappa \models a \in b$ then $\mu \models $f$_{\kappa\mu}(a) \in $ f$_{\kappa\mu}(b)).$ We will leave the details to the reader, and content ourselves with one observation. Since $\lambda$ was pointedly left out of the domain of Sep$_{a,\phi,\lambda}$, for all $\mu \geq \lambda$ including $\lambda$ itself
$$Sep_{a,\phi,\lambda}(\mu) = \bigcup_{\nu \in [\kappa, lambda)} f_{\nu\mu}"Sep_{a,\phi,\lambda}(\nu).
$$

Notice that {\bf M} (i.e. $\kappa \models \phi$) is definable in {\bf V}[G]. In particular, as in the previous section, $\Delta_{0}$-satisfaction is $\Delta_{1}$. Furthermore, since G is set-generic over {\bf V}, by general forcing technology, p $\mid\vdash ``\kappa \models \phi"$ is definable in {\bf V}, and, in particular, p $\mid\vdash ``\kappa \models \phi", \; \phi \;     \Delta_{0},$ is $\Delta_{1}$.

\begin{theorem}
{\bf M} $\models$ CZF - Subset Collection + Exponentiation + $\neg$ Subset Collection
\end{theorem}

\begin{proof}
Extensionality, Set Induction, and Infinity are left to the reader.

Union: This follows from the presence (and meaning) of the terms of the form Union($a$).

Pairing: Given $a, b \in {\bf M}_{\kappa}$, let $\lambda$ include dom($a$) and dom($b$). For $\mu \in \lambda \backslash \kappa$ let f($\mu$) be \{f$_{\kappa\mu}(a)$, f$_{\kappa\mu}(b)$\}, which suffices. Notice that f is definable even if $a$ or $b$ is just a term, so that f exists at least as a term in {\bf V} (even if inductively uninterpretable in {\bf V}).

Restricted Separation: This follows from the presence (and meaning) of the terms of the form Sep$_{a,\phi,\lambda}$, with the proof being essentially that in the previous section. The one change is that the isomorphism given there must be extended to account for terms of the form G$\kappa$ (the other terms already being accounted for inductively). Under this isomorphism, G$_{\lambda+\xi}$ is to be sent to G$_{\mu+\xi}$.

In fact, full Separation holds, even though the separation terms are only for $\Delta_{0}$ formulas. The first guess as to how to proceed is as follows. Given $a$ and $\phi$ a term and formula in (resp. over) {\bf M}$_{\kappa}$, let $\lambda > \kappa$ include the domains of $a$ and of $\phi$'s parameters. Let $\mu > \lambda$ be such that {\bf V}$_{\mu}$[G] contains $a$, $\phi$'s parameters, and is a $\Sigma_{n}$ elementary substructure of {\bf V}[G] for n sufficiently large (i.e. large enough to reflect the truth of $\phi(x)$ in {\bf M}). Let f be the function with domain $\mu \backslash \kappa$ such that f($\nu$) = \{ g $\in$ {\bf V}$_{\mu} \mid$ g is a term naming an element of {\bf M}$_{\nu}$ \}. The notation ``$\phi^f$" means ``$\phi$ with all of the unbounded variables bounded by f". By the choice of $\mu$, $\phi(x)$ is interpreted in {\bf M} (in {\bf V}[G]) just as $\phi(x)$ would be in {\bf V}$_{\mu}$[G]'s version of {\bf M}. If in turn that could be captured as Sep$_{a,\phi^f,\mu}$ (in {\bf V}[G]), then we'd be done.

The problem is that, while f reflects the restriction to {\bf M}$^{\mu}$, it cannot reflect the restriction of {\it P} from ORD to $\mu$. To give a concrete example, let $\phi$ be the sentence $``\forall x \; \forall y \; x=y \vee x\not=y"$. In {\bf M}, this sentence is never true, and so any separation set for $\phi$ should be the empty set. However, for any choice of f, $\phi^f$ eventually becomes true (at latest beyond f's domain), so Sep$_{a,\phi^f,\mu}$ eventually becomes $a$.

We need some way of saying ``don't look beyond this node." This can be done with a technique developed in its essence by Friedman and Scedrov. In \cite{FS}, they consider a particular formula $\psi$ (in the notation to be introduced) and a particular Kripke model; moreover, their model was a two-node model, so their {\bf M}$_{pre-\psi}$ (see below) was a classical and not even a Kripke structure. Nonetheless, their work is easily generalized to all formulas and all models. To have this generalization shown, and to keep this paper self-contained, we develop the details.

Let $\psi$ be a sentence, possibly with parameters from the universe at the bottom node $\bot$ in some Kripke model {\bf M}. Using $\psi$, a transformation on the formulas can be defined inductively. If $\phi$ is atomic (including $\bot$, the always false statement), then $\phi^{*}$ is $\phi \vee \psi$. The inductive clauses are transparent to this construction: $(\phi_{0} \rightarrow \phi_{1})^{*} = (\phi_{0}^{*} \rightarrow \phi_{1}^{*}), \forall x \; \phi(x) = \forall x \; \phi^{*}(x),$ etc. It is easy to see inductively on $\phi$ that for all formulas $\phi$ and all nodes $\sigma$ with $\sigma \models \psi, \sigma \models \phi^{*}$. Let {\bf M}$_{pre-\psi}$ be {\bf M} with all the nodes where $\psi$ holds thrown out. That is, {\bf M}$_{pre-\psi}$'s underlying partial order is the initial sub-order of {\bf M}'s {\it P} consisting of $\{ \sigma | \sigma \not \models \psi \}$ =$_{def}$ {\it P}$_{pre-\psi}$ (which is assumed to be non-empty, i.e. $\bot \not \models \psi$). The universe of {\bf M}$_{pre-\psi}$ at $\sigma$ is the same as for {\bf M}, the transition functions are the same, the atomic relations and functions have the same interpretation. Because nodes in ${\bf M}_{pre-\psi}$ are also nodes in {\bf M}, we will distinguish notationally truth in ${\bf M}_{pre-\psi}$ from truth in {\bf M} with $\models_{{\bf M}_{pre-\psi}}$ versus $\models_{{\bf M}}$.

\begin{lemma}
$\sigma \models_{{\bf M}} \phi^{*}$ iff $\sigma \models_{{\bf M}_{pre-\psi}} \phi$.
\end{lemma}
\begin{proof}
By induction on $\phi$.
\begin{itemize}

\item $\phi$ atomic: $\sigma \models_{{\bf M}} \phi^{*}$ iff $\sigma \models_{{\bf M}} \phi \vee \psi$. Since by hypothesis $\sigma \not \models_{{\bf M}} \psi$, this is equivalent to $\sigma \models_{{\bf M}} \phi$, which is itself equivalent to $\sigma \models_{{\bf M}_{pre-\psi}} \phi$, since {\bf M}$_{pre-\psi}$ is a restriction of {\bf M} and the atomic relations are defined locally. (Although this lemma is very general, the reader can be excused for thinking of models of set theory, and ask whether equality as defined by ``same members" really holds. Note we do not claim that Extensionality is preserved! The equality in {\bf M} is taken as primitive.)

\item $\phi_{0} \vee \phi_{1}$: $\sigma \models_{{\bf M}} \phi^{*}_{0} \vee \phi^{*}_{1}$ iff\\
$\sigma \models_{{\bf M}} \phi^{*}_{0}$ or $\sigma \models_{{\bf M}} \phi^{*}_{1}$ iff (inductively)\\
$\sigma \models_{{\bf M}_{pre-\psi}} \phi_{0}$ or $\sigma \models_{{\bf M}_{pre-\psi}} \phi_{1}$ iff \\
$\sigma \models_{{\bf M}_{pre-\psi}} \phi_{0} \vee \phi_{1}$ .

\item $\phi_{0} \wedge \phi_{1}$: Analogous to $\vee$.

\item $\phi_{0} \rightarrow \phi_{1}$: It has already been observed that:\\
for $\tau \in {\it P} \backslash {\it P}_{pre-\psi} \tau \models_{{\bf M}} \phi^{*}_{0} \wedge \phi^{*}_{1}.$\\
So for $\sigma \in$ {\it P}$_{pre-\psi}$, $\sigma \models_{{\bf M}} \phi^{*}_{0} \rightarrow \phi^{*}_{1}$ reduces to:\\
for all extensions $\tau \in$ {\it P}$_{pre-\psi}$ of $\sigma$, if $\tau \models_{{\bf M}} \phi^{*}_{0}$ then $\tau \models_{{\bf M}} \phi^{*}_{1}$.\\
Inductively, this is equivalent to:\\
for all extensions $\tau \in$ {\it P}$_{pre-\psi}$ of $\sigma$, if $\tau \models_{{\bf M}_{pre-\psi}} \phi_{0}$ then $\tau \models_{{\bf M}_{pre-\psi}} \phi_{1}$, i.e.
$\sigma \models_{{\bf M}_{pre-\psi}} \phi_{0} \rightarrow \phi_{1}.$

\item $\forall x \; \phi(x)$: It has already been observed that:\\
for $\tau \in {\it P} \backslash {\it P}_{pre-\psi}$ and $x \in {\bf M}_{\tau},  \tau \models_{{\bf M}} \phi^{*}(x)$.\\
So for $\sigma \in$ {\it P}$_{pre-\psi}$, $\sigma \models_{{\bf M}} \forall x \; \phi^{*}(x)$ reduces to:\\
for all extensions $\tau \in$ {\it P}$_{pre-\psi}$ of $\sigma$ and $x \in {\bf M}_{\tau},  \tau \models_{{\bf M}} \phi^{*}(x)$.\\
Inductively, this is equivalent to:\\
for all extensions $\tau \in$ {\it P}$_{pre-\psi}$ of $\sigma$, and $x \in {\bf M}_{\tau},  \tau \models_{{\bf M}_{pre-\psi}} \phi(x),$\\
i.e. $\sigma \models_{{\bf M}_{pre-\psi}} \forall x \; \phi(x).$

\item $\exists x \; \phi(x)$: $\sigma \models_{{\bf M}} \exists x \; \phi^{*}(x)$ iff\\
for some $x \in {\bf M}_{\sigma} \; \sigma \models_{{\bf M}} \phi^{*}(x)$ iff (inductively)\\
for some $x \in {\bf M}_{\sigma} \; \sigma \models_{{\bf M}_{pre-\psi}} \phi(x)$ iff\\
$\sigma \models_{{\bf M}_{pre-\psi}} \exists x \; \phi(x)$.
\end{itemize}
\end{proof}

To return to the issue at hand, let $\psi$ be the sentence ``0 $\in 1_{\mu}$", where $\mu$ is as above. (Recall that $1_{\mu}$ is defined so that $\psi$ is true exactly at nodes $\mu$ and beyond.) This choice of $\psi$ determines a *-transformation on formulas. Then Sep$_{a,\phi^{f*},\mu}$ (equivalently, Sep$_{a,\phi^{*f},\mu}$) is as desired. That is, we claim that
$$\kappa \models \forall x \; (x \in Sep_{a,\phi^{*f},\mu} \leftrightarrow (x \in a \wedge \phi(x))).
$$
(In what follows, we suppress mention of the transition functions f$_{\kappa\lambda}$, for purposes of readability, preferring instead to write the incorrect ``$\lambda \models x \in a$" instead of ``$\lambda \models x \in$ f$_{\kappa\lambda}(a)$", to give only a mild example of what the reader would otherwise be in for.)

First suppose that $\kappa \leq \lambda < \mu$ and $x \in$ {\bf M}$_{\lambda}$. (In this paragraph, we suppress mention of the transition functions f$_{\kappa\lambda}$, for purposes of readability, preferring instead to write the incorrect ``$\lambda \models x \in a$" instead of ``$\lambda \models x \in$ f$_{\kappa\lambda}(a)$", to give only a mild example of what the reader would otherwise be in for.) If $x \not\in$ {\bf M}$^{\mu}$ then $x$ gets into neither $a$ nor Sep$_{a,\phi^{*f},\mu}$, and we're done. Otherwise, by the choice of $\mu$,
$$\lambda \models ``x \in a \; \wedge \; \phi(x)" \; iff \; {\bf V}_{\mu}[G] \models ``\lambda \models ``x \in a \; \wedge \phi(x)" ".$$
The difference between {\bf M} (in {\bf V}[G]) and {\bf V}$_{\mu}$[G]'s version thereof is the partial order (ORD vs. $\mu$) and the ordinal length of the inductive generation of terms (again, ORD vs. $\mu$). The former is captured by $\psi$, the latter by f; notationally, {\bf M} as interpreted in {\bf V}$_{\mu}$[G] is {\bf M}$^{\mu}_{pre-\psi}$. So
$${\bf V}_{\mu}[G] \models ``\lambda \models ``x \in a \; \wedge \phi(x)" " \; iff \;
\lambda \models_{{\bf M}_{pre-\psi}} ``x \in a \; \wedge \phi^f(x)".$$
By the previous lemma, the latter is equivalent to $\lambda \models ``x \in a \; \wedge \phi^{f*}(x)".$ By the interpretation of Sep-terms, the latter is equivalent to $\lambda \models ``x \in$ Sep$_{a,\phi^{f*},\mu}"$.

Now consider the case $\lambda \geq \mu$, $x \in {\bf M}_{\lambda}$. Since $\lambda \geq \mu >$ dom($a$), if $\lambda \models ``x \in a"$ then there are $\nu < \mu$ and $y \in {\bf M}_{\nu}$ such that $x$ = f$_{\nu\mu}(y)$ and $\nu \models ``y \in a"$. Withough loss of generality $\nu$ can be chosen to be larger than the domains of $a$ and $\phi$'s parameters. For any such $\nu$ and $y$, by the same isomorphism argument as in the last section, $\nu \models \phi(y)$ iff $\lambda \models \phi(x)$. Furthermore, since $a \in {\bf V}_{\mu}$[G] (by the choice of $\mu$), $y$ is not only in {\bf M}$_{\nu}$ but also {\bf M}$_{\nu}^{\mu}$. Combining all of this with the definition of f$_{\kappa\lambda}$(Sep$_{a,\phi^{f*},\mu}$) (which was highlighted just before the statement of the current theorem), we get the following equivalences:

$\lambda \models ``x \in a \; \wedge \; \phi(x)" \longleftrightarrow$

$\exists \nu < \mu, \; y \in {\bf M}_{\nu}^{\mu} \; (x = f_{\nu\lambda}(y) \; \wedge \; \nu \models ``y \in a \; \wedge \; \phi(y)") \longleftrightarrow$

(by the case above: $\kappa \leq \lambda < \mu$, with $\nu$ here playing the role of $\lambda$)

$\exists \nu < \mu, \; y \in {\bf M}_{\nu}^{\mu} \; (x = f_{\nu\lambda}(y) \; \wedge \; \nu \models ``y \in$ Sep$_{a,\phi^{f*},\mu}") \longleftrightarrow$

$\exists \nu < \mu, \; y \in {\bf M}_{\nu} \; (x = f_{\nu\lambda}(y) \; \wedge \; \nu \models ``y \in$ Sep$_{a,\phi^{f*},\mu}") \longleftrightarrow$

$\lambda \models ``x \in$ Sep$_{a,\phi^{f*},\mu}"$.

Strong Collection: In the presence of full Separation (see last paragraph), Strong Collection follows from Collection. Suppose $\kappa \models \forall x \in a \; \exists y \; \phi(x,y)$. Let $\lambda, \mu$ and f be as in the last paragraph. Then $\kappa \models$ ``f is a bounding set for $\phi$ on $a$".

Exponentiation: In order to highlight exactly why the following argument works for functions (yielding Exponentiation) but not for arbitrary relations (hence not yielding Subset Collection), it will be presented in terms of a relation R until exactly that moment when we need to use that R is not only a relation but also a function.

So let $a$ and $b$ be terms for sets in {\bf M}$_{\kappa}$, and p a forcing condition, R a term, and $\lambda \geq \kappa$ an ordinal so that p $\mid\vdash ``\lambda \models$ R is a relation from $a$ to $b$". Since the property of being a relation between two given sets is $\Delta_{0}$, the assertion p $\mid\vdash ``\lambda \models$ R is a relation from $a$ to $b$" is $\Delta_{1}$ in {\bf V}. We will use later the same result for the property of R being a function from $a$ to $b$.

Let H be the $\Sigma_{1}$-Skolem hull of \{$a$, $b$, $\kappa$, $\lambda$, R\} $\cup$ TC($a$) $\cup$ TC($b$) $\cup \; \kappa$, of the same size as that set (where ``TC" stands for transitive closure). Let $\rho$ be the transitive collapsing function. Then
$$\rho"H \models ``\rho(p) \mid\vdash ``\rho(\lambda) \models \rho(R) \; is \; a \; relation \; from \; \rho(a) \; to \; \rho(b)"".
$$
$\rho$ fixes p, $a$, and $b$ because L$_{\omega}$, TC($a$), and TC($b$) $\subseteq$ H. Furthermore, $\rho(\lambda) \leq \lambda$, so for some $\mu \leq \lambda$
$$\rho"H \models ``p \mid\vdash ``\mu \models \rho(R) \; is \; a \; relation \; from \; a \; to \; b"",
$$
and, by the same reasoning, for all q $\geq$ p,
$$H \models ``q \mid\vdash ``\lambda \models \langle x_{a}, y_{b} \rangle \in R"" \; iff \; \rho"H \models ``q \mid\vdash ``\mu \models \langle x_{a}, y_{b} \rangle \in \rho(R)"".
$$
However, the assertion ``q $\mid\vdash ``\mu \models \langle x_{a}, y_{b} \rangle \in \rho(R)""$ is $\Delta_{1}$ over both {\bf V} and $\rho$"H (as the collapse of a $\Sigma_{1}$ substructure of {\bf V}), hence is absolute (between $\rho$"H and {\bf V}), as is $``q \mid\vdash ``\lambda \models \langle x_{a}, y_{b} \rangle \in R""$ between H and {\bf V}. Therefore
$$q \mid\vdash ``\lambda \models \langle x_{a}, y_{b} \rangle \in R" \; iff \; q \mid\vdash ``\mu \models \langle x_{a}, y_{b} \rangle \in \rho(R)".
$$
Since $\mu \leq \lambda$
$$if \; q \mid\vdash ``\mu \models \langle x_{a}, y_{b} \rangle \in \rho(R)" \; then \; q \mid\vdash ``\lambda \models \langle x_{a}, y_{b} \rangle \in \rho(R)".
$$
Combining the last two assertions,
$$if \; q \mid\vdash ``\lambda \models \langle x_{a}, y_{b} \rangle \in R" \; then \; q \mid\vdash ``\lambda \models \langle x_{a}, y_{b} \rangle \in \rho(R)",
$$
i.e. p $\mid\vdash ``\lambda \models$ R $\subseteq$ $\rho$(R)".

Under the assumption merely that R is a relation, no more can be said. If, however, p $\mid\vdash ``\lambda \models$ R is a (total) function (from $a$ to $b$)", then, by analogous arguments to those above,

H $\models$ ``p $\mid\vdash ``\lambda \models$ R is a function"",

$\rho$(H) $\models$ ``p $\mid\vdash ``\mu \models$ $\rho$(R) is a function"",

p $\mid\vdash ``\mu \models$ $\rho$(R) is a function", and

p $\mid\vdash ``\lambda \models$ $\rho$(R) is a function".

Since a (total) function cannot be properly extended to a function, p $\mid\vdash ``R = \rho(R)"$.

By the cardinality constraint on H, there is an ordinal $\alpha$ (independent of $\lambda$ and R) such that $\rho$"H $\subseteq$ {\bf V}$_{\alpha}$. In particular, $\rho$(R) $\in$ {\bf V}$_{\alpha}$. Let g be the function with domain $\alpha \backslash \kappa$ such that g($\nu$) = \{ h $\in$ {\bf V}$_{\alpha} \mid$ h is a term naming an element of {\bf M}$_{\nu}$ \}. Then g is a bounding set for $^{a}b$. By Separation (Restricted Separation suffices here), $^{a}b$ is a set.

$\neg$ Subset Collection: Suppose that $\kappa \models$ ``for all total relations R on {\bf N} there is a total subrelation R', R' $\in$ F". There is an ordinal $\lambda$ such that, for all terms G$_{\mu}$ in the transitive closure of F, $\mu < \lambda$; without loss of generality, $\lambda > \kappa$. Since for every such $\mu$ $\lambda \models$ G$_{\mu}$ = {\bf N} $\times$ {\bf N}, f$_{\kappa\lambda}$(F) can be evaluated in {\bf V}. By assumption, $\lambda \models$ ``there is a total subrelation of G$_{\lambda}$ in f$_{\kappa\lambda}$(F)". This is, however, impossible, by the genericity of G, as follows. Let $\lambda \models$ ``R' $\in$ f$_{\kappa\lambda}$(F)". In {\bf V}, R'($\lambda$) can be identified with a total relation on {\bf N} (and so we will use the same notation R'($\lambda$) for this relation in {\bf V}). Let p be a forcing condition. Since p is finite and R'($\lambda$) total, let $\langle m, n \rangle \in$ R'($\lambda$) be such that $\langle m, n \rangle \not\in$ dom(p), and let q be p $\cup \langle \langle m, n \rangle, 0 \rangle$. q $\mid\vdash$ ``R'($\lambda$) is not a subrelation of G", so q $\mid\vdash ``\lambda \not\models$ ``R' is a subrelation of G$_{\lambda}$"". Hence the set of conditions forcing $\lambda \not\models$ ``R' is a subrelation of G$_{\lambda}$" is dense, and so in {\bf V}[G] $\lambda \not\models$ ``R' is a subrelation of G$_{\lambda}$". This however contradicts the construction of $\lambda$, so the assumption on $\kappa$ is false, and Subset Collection fails for relations on {\bf N}.
\end{proof}

\section{Reflection}
Although this article is about Power Set and its variants, a result about Collection and its variants is right there, and so should be pointed out. Classically, over the other axioms of ZF, Reflection, Collection, and Replacement are equivalent. The proofs that Reflection implies Collection, which in turn implies Replacement are soft and easy proofs, and so carry over verbatim to intuitionistic logic. The reverse implications are less obvious and make essential use of classical logic; the standard proofs do not carry over. In fact, it was shown in \cite{FS} that Replacement does not imply Collection over the remaining axioms of IZF. It is natural to ask whether the same can be said of Collection and Reflection. Although this is still unknown, the first construction of this paper provides a model of CZF + Separation + $\neg$ Reflection. The only axiom that's missing from IZF is Power Set, and Subset Collection is at least a partial substitute.

Moreover, the failure of Reflection is very strong. All true $\Pi_{1}$ facts are reflected in any transitive set, and any true $\Sigma_{1}$ sentence is reflected in any transitive set that contains a witness. In contrast, consider $\neg \forall x \; \forall y \; (x=y \vee x\not=y)$. This is true in the whole model, and false in every set.

\end{document}